\documentclass{amsart}

\usepackage{amssymb}
\usepackage{graphicx}
\usepackage{MnSymbol}

\title{Torsion-free Word Hyperbolic Groups are Noncommutatively Slender}
\author{Samuel M. Corson}

\theoremstyle{definition}\newtheorem{theorem}{Theorem}
\theoremstyle{definition}
\theoremstyle{definition}
\theoremstyle{definition}\newtheorem{definition}[theorem]{Definition}
\theoremstyle{definition}
\theoremstyle{definition}
\theoremstyle{definition}
\theoremstyle{definition}
\theoremstyle{definition}\newtheorem{lemma}[theorem]{Lemma}
\theoremstyle{definition}
\theoremstyle{definition}
\theoremstyle{definition}
\theoremstyle{definition}
\newtheorem*{question*}{Question}

\newcommand{\W}{\mathcal{W}}
\newcommand{\HEG}{\textbf{HEG}}

\begin{document}

\address{Mathematics Department\\
1326 Stevenson Center\\
Vanderbilt University\\
Nashville, TN 37240\\
USA}			

\email{samuel.m.corson@vanderbilt.edu}

\maketitle

\begin{abstract}  In this note we prove the claim given in the title.  A group $G$ is noncommutatively slender if each map from the fundamental group of the Hawaiian Earring to $G$ factors through projection to a canonical free subgroup.  Graham Higman, in his seminal 1952 paper \cite{H}, proved that free groups are noncommutatively slender.  Such groups were first defined by K. Eda in \cite{E}.  Eda has asked which finitely presented groups are noncommutatively slender. This result demonstrates that random finitely presented groups in the few-relator sense of Gromov are noncommutatively slender.
\end{abstract}

\begin{section}{Introduction}  Certain groups that allow for infinite multiplication exhibit a curious behavior, namely that maps to particular well understood groups are always boring.  This phenomenon was first noticed by Specker in \cite{S}, who proved that for each integer valued homomorphism from the countable product of integers $\phi: \prod_{\omega} \mathbb{Z} \rightarrow \mathbb{Z}$ there exists a natural number $N$ such that the projection $p_N: \prod_{\omega} \mathbb{Z} \rightarrow \prod_{n=0}^{N} \mathbb{Z}$ satisfies $\phi = \phi \circ p_N$.  The phenomenon was studied by \L os and lead him to define a slender group to be a torsion-free abelian group $A$ for which any homomorphism $\phi: \prod_{\omega} \mathbb{Z} \rightarrow A$ has an $N \in \omega$ for which $\phi = \phi\circ p_N$.  These groups have been extensively studied and also classified via subgroups (see for example \cite{F} volume 2, Sections 94, 95).

The term noncommutatively slender (we will use the contraction n-slender to be short) was introduced by K. Eda. The first examples of such groups were free groups as demonstrated by Higman (in \cite{H}) 40 years before Eda defined such groups.  Eda's idea is essentially the same as with the abelian case, with the domain being replaced by the fundamental group of the Hawaiian Earring.  He showed that n-slender groups are necessarily torsion-free and all abelian n-slender groups are indeed slender in the abelian sense.  Also, the class of n-slender groups is closed under direct sums and free products (see \cite{E} for the definition and such results).  The Hawaiian Earring group can be used to endow a group with an infinite multiplicative structure, thus n-slender groups are those which resist such a structure.

In contrast to slender groups, no nice characterization for n-slender groups via subgroups is known.  Eda has noted that n-slenderness is an open question even for finitely presented groups \cite{E'}.  This note gives a broad class of finitely presented groups which are n-slender, and shows that most finitely generated groups are n-slender in the few-relator sense of Gromov \cite{G}.  The theorem is the following:

\begin{theorem}\label{main}  If $G$ is a a torsion-free word hyperbolic group then $G$ is n-slender.
\end{theorem}

The hypothesis torsion-free cannot be dropped, as any group with torsion fails to be n-slender (\cite{CS} provides some interesting examples of homomorphisms to torsion).  The proof of the theorem uses an interesting theorem about torsion-free hyperbolic groups and a modification of a theorem in \cite{H}.  In Section 2 some background definitions and results are provided.  We define the universal monotone condition (u.m.), which is that there exists a length function on the group such that for any real number $r$ there exists a power $K_r$ such that $g^{K_r}$ is of length at least $r$ for any $g \neq 1$.  We show that this condition is sufficient to imply that a group is n-slender.  In Section 3 we prove torsion-free hyperbolic groups are u.m. which concludes the proof of Theorem \ref{main}.

In Section 4 we prove the following theorem, which generalizes the fact that n-slender groups are closed under direct sums and direct products:

\begin{theorem}\label{Graphprod}  The class of n-slender groups is closed under taking graph products.
\end{theorem}

%

In Section 5 we motivate the question of whether Thompson's group $F$ is n-slender.  We also show that the natural length function defined by the characterization of $F$ as a diagram group is not u.m.  Finally in Section 6 we give a family of examples to show that even very uncompicated n-slender groups can fail to be u.m.

\end{section}

\begin{section}{Non-commutatively Slender Groups}

In this section we give a characterization of the Hawaiian Earring fundamental group, define n-slenderness, and prove a modification of Higman's original theorem of the n-slenderness of free groups.

The Hawaiian Earring is a shrinking wedge of countably-infinitely-many circles.  More formally, given a point $p\in \mathbb{R}^2$ and $r\in (0, \infty)$ we let $C(p, r)$ denote the circle centered at $p$ of radius $r$ and define the Hawaiian Earring to be the subspace $E=\bigcup_{n\in \omega} C((0, \frac{1}{n+2}),\frac{1}{n+2})$ of $\mathbb{R}^2$.  Naively the fundamental group of $E$ might seem to be a free group of countably infinite rank (one free generator for each circle in the union defining $E$), but the fundamental group is in fact uncountable.  We give a combinatorial characterization of this group using countable words.

Let $\{a_n^{\pm 1}\}_{n=0}^{\infty}$ be a countably infinite set with formal inverses, the elements of which we call \textbf{letters}.  A map $W: \overline{W}\rightarrow \{a_n^{\pm 1}\}_{n=0}^{\infty}$ from a countable totally ordered set $\overline{W}$ is a \textbf{word} if for each $n\in \omega$ the set $W^{-1}(\{a_n^{\pm 1}\})$ is finite.  Two words $U$ and $V$ are isomorphic, denoted $U \simeq V$, if there exists an order isomorphism of the domains of each word $f: \overline{U} \rightarrow \overline{V}$ such that $U(t) = V(f(t))$.  We identify isomorphic words.  The class of isomorphic words is a set of cardinality continuum which we denote $\W$.  For each $N\in \omega$ there is a projection map $p_N$ to the set of finite words given by letting $p_N(W) = W|\{t\in \overline{W}: W(t) \in \{a_n^{\pm 1}\}_{n=0}^{N}\}$.  For words $U, V\in \W$ we let $U \sim V$ if for each $N\in \omega$ we have that $p_N(U) = p_N(V)$ in the free group $F(\{a_0, \ldots, a_N\})$.  This is an equivalence relation.  For each word $U$ there is an inverse word $U^{-1}$ whose domain is the totally ordered set $\overline{U}$ under the reverse order and $U^{-1}(t) = U(t)^{-1}$.  Given two words $U, V\in \W$ there is a natural way to form the concatenation $UV$.  In particular, one takes the domain of $UV$ to be the disjoint union of $\overline{U}$ with $\overline{V}$, with order extending that of $\overline{U}$ and $\overline{V}$ and placing all elements of $\overline{U}$ before those of $\overline{V}$, and $UV(t) = \begin{cases}U(t)$ if $t\in \overline{U}\\V(t)$ if $t\in \overline{V}  \end{cases}$.  The set $\W/\sim$ now has a group structure with binary operation given by $[U][V] = [UV]$, inverses defined by $[U]^{-1} = [U^{-1}]$ and the trivial element given by the equivalence class of the empty word.

Let $\HEG$ denote the group $\W/\sim$.  The free group $F(\{a_0, \ldots, a_N\})$, which we shall denote $\HEG_N$, may be though of as a subgroup in $\HEG$ in the obvious way.  Moreover, the word map $p_N$ defines a group retraction $\HEG \rightarrow   \HEG_N$ which we denote $p_N$ by abuse of notation.  There is another word map $p^N$ given by $p^N(W) = W|\{t\in \overline{W}: W(t) \in \{a_n^{\pm 1}\}_{n=N+1}^{\infty}\}$ which gives a group retraction from $\HEG$ to the subgroup $\HEG^N$ consisting of those equivalence classes which contain words involving no letters in $\{a_n^{\pm 1}\}_{n=0}^N$.  We again abuse notation by calling this retraction $p^N$.  There is a canonical isomorphism $\HEG \simeq \HEG_N \ast \HEG^N$ obtained by considering a word $W$ as a concatenation of finitely many words in the letters $\{a_n^{\pm 1}\}_{n=0}^N$ and finitely many words in the letters $\{a_n^{\pm 1}\}_{n=N+1}^{\infty}$.

\begin{definition}  A group $G$ is noncommutatively slender (or n-slender) if for each homomorphism $\phi: \HEG \rightarrow G$ there exists $N\in \omega$ such that $\phi = \phi\circ p_N$.
\end{definition}

In other words, $G$ is n-slender if for each homomorphism $\phi:\HEG \rightarrow G$ there exists $N$ so that the restriction of $\phi$ to $\HEG^N$ is the trivial homomorphism.

For our purposes a \textbf{length function} on a group $G$ is a map $l: G \rightarrow \mathbb{R}$ such that the following hold:

\begin{enumerate}

\item $l(g) \geq 0$ with $l(1) = 0$

\item $l(g) = l(g^{-1})$

\item $l(gh) \leq l(g) + l(h)$

\end{enumerate}

As an example, if $G$ has a generating set $X$ one may define a length function $l_X$ on $G$ by letting $l_X(g)$ be the distance from $1$ to $g$ in the Cayley graph $\Gamma(G, X)$.  In other words, $l_X(g)$ is the length of a minimal word in the generators $X$ that is needed to represent the group element $g$ in $G$.  We say that a length function $l$ is \textbf{universally monotone} if for each $r\in \mathbb{R}$ there exists $K_r \in \omega$ such that for each $g\in G-\{1\}$ we have that $l(g^{K_r}) \geq l(g) + r$.  In particular, for a universally monotone length function we have $l(g) = 0$ if and only if $g=1$.  We say a group is \textbf{universally monotone} (or \textbf{u.m.}) if it has a length function which is universally monotone.

The following is the critical theorem, the ideas of which are in \cite{H}:

\begin{theorem}\label{u.m.} If the group $G$ is u.m. then $G$ is n-slender.
\end{theorem}

\begin{proof}  Let $l$ be a universally monotone length function on $G$.  Let $\phi:\HEG \rightarrow G$ be a homomorphism and suppose for contradiction that the restriction of $\phi$ to each $\HEG^N$ is nontrivial.  Select a sequence of words $\{W_m\}_{m=0}^{\infty}$ such that $W_m$ uses only letters in $\{a_n^{\pm 1}\}_{n=m}^{\infty}$ and $\phi([W_m]) \neq 1$.  Let $r_m= l(\phi([W_m]))$ and $k_m\in \omega$ be such that $g\in G-\{1\}$ implies $l(g^{k_m}) \geq l(g)+ r_m +1$.

Consider the word $U = U_0$ defined by the equations $U_{p-1} = W_pU_{p}^{k_m}$.  In other words, one can think of $U$ as being of form $U = W_1(W_2(W_3(\cdots)^{k_3})^{k_2})^{k_1}$.  Notice that if $\phi([U_p]) \neq 1$ we have
\begin{center}
$l(\phi([U_p]^{k_p})) = l((\phi([U_p]))^{k_p}) \geq l(\phi([U_p]))+ r_p +1 = l(\phi([U_p])) + l(\phi([W_p])) +1$

\end{center}

from which we have

\begin{center}  $l(\phi([U_{p-1}])) = l(\phi([W_p])\phi([U_p]^{k_p})) \geq l(\phi([U_p]^{k_p})) - l(\phi([W_p])) \geq l(\phi([U_p]))+1$
\end{center}

Thus $\phi([U_{p-1}]) \neq 1$ and the argument may be repeated.  By induction we get that $l(\phi([U_{p-p'}])) \geq l(\phi([U_{p}])) + p'$ for $p'\leq p$.  Thus if $p> l(\phi([U_0]))$ we have $\phi([U_p])=1$, which gives $1 = \phi([U_p]) = \phi([W_{p+1}])(\phi([U_{p+1}]))^{k_{p+1}} = \phi([W_{p+1}])$, a contradiction.

\end{proof}

\end{section}

\begin{section}{Hyperbolic Groups}

We review some basic concepts related to hyperbolic groups and prove that every torsion-free word hyperbolic group is u.m.

Recall that a metric space $(Z, d)$ is \textbf{hyperbolic} if there exists a $\delta$ such that for all $p, x,y,z\in Z$ we have $$(x,z)_p \geq \min\{(x,y)_p, (y,z)_p\} - \delta$$ where $(x,y)_p = \frac{1}{2}(d(x,p) + d(y,p) - d(x,y))$ is the Gromov product.  A geodesic metric space $(Z, d)$ is hyperbolic if and only if there exists a $\delta$ such that for all points $x,y,z \in S$, and geodesics $[x,y]$ and $[x,z]$,  the points $v\in [x,y]$ and $w\in [x,z]$ satisfying $d(x, v) = d(x,w) = (y,z)_x$ also satisfy $d(v,w) \leq 2\delta$.  The $\delta$ used in the alternative criterion for geodesic spaces is not necessarily the same as in the original definition.  Bounded spaces and the classical hyperbolic metric spaces $\mathbb{H}^n$ are examples of hyperbolic spaces.

A finitely generated group $G$ is word hyperbolic if for some finite generating set $X$ the Cayley graph $\Gamma(G, X)$ is a hyperbolic space under the combinatorial path metric (under which $\Gamma(G, X)$ is a geodesic space).  It turns out that for a hyperbolic group $G$ it is the case that for any finite generating set $X$ the Cayley graph $\Gamma(G, X)$ is hyperbolic.  

Now we fix some notation.  Let $G$ be a group with generating set $X$.  Any word $W$ in the letters $X^{\pm 1}$ gives an element of the group $G$ by performing the necessary multiplication of the letters.  Write $W=_G g$ if the word $W$ represents the element $g\in G$ and $W=_G U$ if the words $W$ and $U$ represent the same element in $G$.  Let $\|W\|$ denote the length of a word $W$ in the letters $X^{\pm 1}$.  Let $l_X$ be the length function induced by $X$ on $G$, that is $l_X(g) = \min\{\|W\|: W =_G g\}$, and by abuse of notation let $l_X(W) = l_X(g)$ where $W=_G g$.  Obviously $l_X(W)\leq \|W\|$.  Given words $V, W$ in the letters $X^{\pm 1}$ we say that $V$ is $W$-periodic if $V$ is a subword of a power of $W^{\pm 1}$.  We say a word $W$ in $X^{\pm 1}$ is \textbf{cyclically minimal} if the equality $W = VUV^{-1}$ in $G$ implies that $l_X(W) \leq \|U\|$.  For $R\in \mathbb{R}$ let $D(R) = \{g\in G: l_X(g) \leq R\}$.

We use the following two results which appear as Lemmas 21 and 26 respectively in \cite{O}:

\begin{lemma}\label{ol1}  Let $G$ be a word hyperbolic group and $\delta$ be a constant such that for all $p, x,y,z\in Z$ we have $(x,z)_p \geq \min\{(x,y)_p, (y,z)_p\} - \delta$. Let $K\geq 14 \delta$ and $K_1 > 12(K+\delta)$ and suppose that a geodesic n-gon $[x_1, \ldots, x_n]$ satisfies the conditions $d(x_{i-1}, x_i) >K_1$ for $i = 2, \ldots, n$ and $(x_{i-2}, x_i)_{x_{i-1}}<K$ for $i = 3, \ldots, n$ (if $n \geq 3$).  Then the polygonal line $p = [x_1, x_2] \cup [x_2, x_3]\cup \cdots \cup [x_{n-1}, x_n]$ is contained in a $2K$-neighborhood of the side $[x_1, x_n]$ and the side $[x_1, x_n]$ is contained in a $14\delta$-neighborhood of $p$.
\end{lemma}

\begin{lemma}\label{ol2}  For every hyperbolic group $G$ with finite generating set $X$ and every $\theta>0$ there exists a number $C$ such that for every $W$-periodic word $V$, where $W$ is a cyclically minimal word with $\|W\| >C$ it is true that $l_X(V) \geq (1-\theta)\|V\|$.
\end{lemma}

Recall the following classical facts:

\begin{lemma}\label{undistorted}  If $G$ is word hyperbolic, generated by the finite set $X$, and $g\in G$ is of infinite order then the following two conditions hold:

\begin{enumerate}  \item  There exists $\lambda>0$ such that $|n| \leq \lambda l_X(g^n)$ for all $n\in \mathbb{Z}$.

\item  There exists $L \geq 0$ such that each geodesic $\gamma$ in $\Gamma(G, X)$ between two elements in the subgroup $\langle g\rangle$ is within the $L$-neighborhood of $\langle g\rangle$.

\end{enumerate}

\end{lemma}

The unifom monotonicity of word hyperbolic groups follows immediately from the following theorem, which together with Theorem \ref{u.m.} proves Theorem \ref{main}.

\begin{theorem}  If $G$ is a torsion-free word hyperbolic group and $X$ a finite generating set there exists a constant $N\in \omega$ such that if $g\in G-\{1\}$ then $l_X(g^N)>l_X(g)$.
\end{theorem}

\begin{proof}  Fix $\delta$ which satisfies both the original and the geodesic definitions of hyperbolic space.  Let $l_X(\cdot) = l(\cdot)$ for simplicity of notation.  In Lemma \ref{ol2} we let $\theta = \frac{1}{1000}$ and pick $C$ accordingly where without loss of generality $C> 1000\delta$.  For each $h\in D(C) -\{1\}$ pick $\lambda_h, L_h \in \omega-\{0\}$ as in the statement of Lemma \ref{undistorted}.  Let $\lambda = \max\{\lambda_h\}_{h\in D(C)-\{1\}}$ and $L \geq \max\{L_h\}_{h\in D(C) - \{1\}}, C$.  Let $N \geq 100, 241 L \lambda$.  Let $g\in G-\{1\}$.  We treat cases.  In Cases 1a and 1b we use the fact that $N \geq 100$ and in Cases 2a and 2b we use the fact that $N \geq 241 L \lambda$.  Pick $h$ which is conjugate to $g$ and of minimal length.  Pick $x$ of minimal length such that $g = xhx^{-1}$.

\textbf{Case 1a}.  Suppose $l(h)>C$ and $l(x) \leq 7l(h)$.  Then we have that

\begin{center}  $l(g^N) \geq l(h^N) - 2l(x) \geq l(h^N) - 14 l(h)$

$\geq 100(\frac{1}{2}l(h)) -14l(h)$  (here we are using $N \geq 100$ and $\theta \leq \frac{1}{2}$)

$\geq 36l(h) > l(h) + 14 l(h) \geq l(h) + 2l(x) \geq l(g)$
\end{center}

\textbf{Case 1b}.  Suppose $l(h)>C$ and $l(x) > 7l(h)$.  Consider the geodesic n-gon $[x, xh, xh^2, \ldots, xh^N]$, which is an isometric translate of the geodesic n-gon $[1, h, h^2, \ldots, h^N]$.  Notice that

\begin{center}
$(h^i, h^{i+2})_{h^{i+1}}= (1, h^2)_{h}$

$=\frac{1}{2}(l(h) + d(h, h^2) -l(h^2))$

$\leq \frac{1}{2}(2l(h) - (1-\frac{1}{100})2l(h))$

$=\frac{1}{100}l(h) < \frac{1}{50}l(h)$
\end{center}

for $0 \leq i \leq N-2$.  Also, $d(h^i, h^{i+1}) = l(h)$.  Thus letting $K = \frac{1}{50}l(h)$ and $K_1 = \frac{1}{2}l(h)$ in Lemma \ref{ol1}, we have that $[x, xh^N]$ is in the $14\delta$-neighborhood of $[x, xh]\cup [xh, xh^2] \cup \cdots \cup [xh^{N-1}, xh^N]$.

We use Lemma \ref{ol1} again.  Notice that $l(x) \leq l(xh^i)$ for all $i\in \mathbb{Z}$ by the minimality of the length of $x$ (else $(xh^i)h(xh^i)^{-1} = g$ and we have a contradiction).  Letting $v\in [x, xh^N]$ be such that $d(x,v) = (1, xh^N)_x$ we may pick $v'\in [x, xh]\cup [xh, xh^2] \cup \cdots \cup [xh^{N-1}, xh^N]$ such that $d(v, v') \leq 14 \delta$.  For some $0 \leq i \leq N$ we have that $d(v', xh^i) \leq \frac{l(h)}{2}$.  Thus $d(v, xh^i) \leq \frac{l(h)}{2} + 14 \delta$.  Then

\begin{center}  $d(1, v) \geq d(1, xh^i) - d(v, xh^i)$

$\geq l(x) - \frac{l(h)}{2} - 14 \delta$

\end{center}

Letting $w\in [1, x]$ be such that $d(x, w) = (1, xh^N)_x$ we have that $d(w, v) \leq 2\delta$, and so 

\begin{center}  $d(1, w) \geq d(1, v) - d(v,w)$

$\geq l(x) - \frac{l(h)}{2} - 14 \delta -2 \delta$

$= l(x) - \frac{l(h)}{2} - 16 \delta$
\end{center}

Thus $(1, xh^N)_x \leq \frac{l(h)}{2} + 16\delta$.  The similar argument shows that $(x, g^N)_{xh^N}\leq \frac{l(h)}{2} + 16\delta$.  Now letting $K = \frac{l(h)}{2} + 17\delta$ and $K_1 = 7l(h)$, we see that $K_1 = 7l(h) \geq 12(\frac{l(h)}{2} + 18\delta)$ since $l(h)>C> 1000\delta$.  Considering the geodesic quadrangle $[1, x, xh^N, g^N = xh^Nx^{-1}]$ we see by Lemma \ref{ol1} that $[1, x] \cup [x, xh^N] \cup [xh^N, g]$ is in the $2K$-neighborhood of $[1, g^N]$.  Pick $s_0, s_1\in [1, g^N]$ such that $d(x, s_0) \leq 2K = l(h) + 34 \delta$ and $d(xh^N, s_1) \leq  2K = l(h) + 34\delta$.  It is easy to see that $s_0\in [1, s_1] \subseteq [1, g^N]$.

Now 
\begin{center}
$l(g^N)= d(1, s_0) + d(s_0, s_1) + d(s_1, g^N)$

$\geq (l(x) - l(h) - 34 \delta) + (l(h^N) -2l(h)- 68 \delta)   +   (l(x) - l(h) - 34 \delta)$

$= 2l(x) + l(h^N) -4l(h) - 136\delta$

$\geq 2l(x)+ 100(\frac{999}{1000}l(h))-4 l(h)-136 \delta$

$> 2l(x)+ 99l(h))-4 l(h)-136 \delta$

$> 2l(x) + l(h)$

$\geq l(g)$
\end{center}

so that we are done in this case.

\textbf{Case 2a}.  Suppose that $l(h) \leq C$ and $l(x) \leq 60 L$.  Then we have

\begin{center}  $l(g^N) \geq l(h^N) - 2l(x)$

$\geq l(h^N) - 120L$

$\geq \frac{N}{\lambda} -120L$

$> 241L -120 L = L + 120L \geq l(h)+ 2l(x) \geq l(g)$

\end{center}

\textbf{Case 2b}.  Suppose that $l(h)\leq C$ and  $l(x) > 60L$.  Let $v\in [x, xh^N]$ be such that $d(x, v) = (1, xh^N)_x$.  As $[x, xh^N]$ is an isometric translation of $[1, h^N]$ we have that there is some $i\in \mathbb{Z}$ such that $d(v, xh^i)\leq L$.  We know $l(xh^i) \geq l(x)$ by the minimality condition on $l(x)$.  Pick $w\in [1, x]$ such that $d(x, w) = (1, xh^N)_x$, so that $d(v,w) \leq 2\delta$.  Then

\begin{center}  $(1, xh^N)_x = l(x) - d(1, w)$

$\leq l(x) - (d(1, v) - 2\delta)$

$\leq l(x)+ 2\delta - (l(xh^i)-L)$

$\leq l(x) + 2\delta +L - l(x) = 2\delta + L$
\end{center}

The condition $(x, g^N)_{xh^N} \leq 2\delta +L$ is proven similarly.  Now we employ Lemma \ref{ol1} again using $K = 2\delta + 2L$ and $K_1 = 24L + 36 \delta$, so that $[1, x]\cup[x, xh^N]\cup[xh^N, g^N]$ is within the $2K = 4\delta + 4L$-neighborhood of $[1, g^N]$.  Select $s_0, s_1 \in [1, g^N]$ so that $d(x, s_0), d(xh^N, s_1)\leq 4\delta + 4L$.  Now

\begin{center}  $l(g^N) = d(1, s_0) + d(s_0, s_1) + d(s_1, g^N)$

$\geq (l(x) -4\delta - 4L) + (l(h^N) - 8\delta -8L) + (l(x) - 4\delta -4L)$

$\geq 2l(x) + l(h^N) - 32L > 2l(x) + l(h) \geq l(g)$
\end{center}

so we are done in this case as well.

\end{proof}

\end{section}

\begin{section}{Graph Products of n-slender Groups}

We recall the definition of a graph product of groups and some machinery, then prove Theorem \ref{Graphprod}.  Suppose $\Gamma = (V, E)$ is a graph (we allow the sets of vertices and edges to be of arbitrary cardinality but do not allow an edge to connect a vertex to itself) and to each vertex $v\in V$ we associate a group $G_v$.  We call the $G_v$ the generating groups.  The graph product $G = \Gamma(\{G_v\}_{v\in V})$ is defined by taking the free product $\ast_{v\in V} G_v$ and modding out by the normal closure of the set $\{[g_{v_0}, g_{v_1}]\}_{g_{v_0} \in G_{v_0}, g_{v_1} \in G_{v_1}, \{v_0, v_1\}\in E}$.  Thus free products of groups and direct sums of groups are examples of graph products of groups, with the graphs having either no edges or being complete in the respective cases.

Each $G_v$ is a retract subgroup of $G$ and $G$ is generated by the elements of the generating subgroups $G_v$.  Thus each element $g\in G$ has a representation as a word $g =_{G} g_0 g_1g_2\cdots g_{n-1}$ with each $g_i$ in a generating group.  In such a word we call each $g_i$ a syllable.  Given two generating groups $G_{v_0}$ and $G_{v_1}$ it is easy to see that the subgroup $\langle G_{v_0}\cup G_{v_1}\rangle \leq G$ is a retract of $G$ and is either isomorphic to $G_{v_0} \ast G_{v_1}$ or $G_{v_0}\times G_{v_1}$, the first being the case if and only if $\{v_0, v_1\} \notin E$.  Thus for nontrivial elements $g_0\in G_{v_0}$ and $g_1\in G_{v_1}$ we have that $[g_0, g_1] = 1$ if and only if $\{v_0, v_1\} \in E$.

We present some machinery found in \cite{Gre}, where graph products were first introduced.  We say a word $g_0 g_1\cdots g_{n-1}$ in elements of the generating groups is reduced if the following hold:

\begin{enumerate}

\item Each $g_i$ is a nontrivial element in a generating group and $g_i$ and $g_{i+1}$ are in different generating groups for all $0\leq i<n-1$

\item If $i \leq k < j$ and $[g_i, g_{i+1}] = [g_i, g_{i+2}]= \cdots = [g_i, g_k] = 1 = [g_{k+1}, g_j] = [g_{k+2}, g_j] = \cdots = [g_{j-1}, g_j]$ then $g_i$ and $g_j$ are in different generating groups.

\end{enumerate}
  
We say that two reduced words $w_0, w_1$ are equivalent,  if one can obtain $w_1$ from $w_0$ by a permutation of syllables as allowed in the group (i.e. one can permute the syllables $g_i$ and $g_{i+1}$ if and only if $[g_i, g_{i+1}] =1$).  Clearly the equivalence of $w_0$ to $w_1$ implies that $w_0$ and $w_1$ have the same word length and $w_0 =_G w_1$.  Using $\bigcup_{v\in V} G_v$ as a generating set for $G$ we get a length function $l$ on $G$.

The following result combines the statements of Theorem 3.9 and Corollary 3.13 of \cite{Gre}:

\begin{lemma}\label{Green}  Each $g \neq 1$ has a reduced word representation $g =_G g_0g_1\cdots g_{n-1}$ which is unique up to equivalence, with $l(g) = n$.
\end{lemma}

We give a lemma before proving Theorem \ref{Graphprod}.

\begin{lemma}\label{extension}  An n-slender by n-slender group is n-slender.
\end{lemma}

\begin{proof}  Suppose that $1\rightarrow K \rightarrow^{\iota} G \rightarrow^q Q \rightarrow 1$ is a short exact sequence of groups with $K$ and $Q$ n-slender groups, where for simplicity we identify $K$ with its image in $G$.  Let $\phi: \HEG \rightarrow G$ be a homomorphism.  By the n-slenderness of $Q$ we see that for $q\circ \phi$ there exists an $N'\in \omega$ such that $\phi|\HEG^{N'}$ maps into the kernel of $q$.  In other words, $\phi$ maps $\HEG^{N'}$ into $K$.  As $\HEG \simeq \HEG_{N'}\ast \HEG^{N'}$ we may define a homomorphism $\phi':\HEG \rightarrow G$ by letting $\phi'|\HEG_{N'}$ be the trivial map and $\phi'|\HEG^{N'} = \phi|\HEG^{N'}$.  By n-slenderness of $K$ there exists an $N \in \omega$, without loss of generality $N >N'$ such that $\phi'|\HEG^N$ is the trivial map.  Then $\phi|\HEG^N$ is the trivial map.
\end{proof}

\begin{proof}{(of Theorem \ref{Graphprod})}  Let $G_v$ be n-slender for each $v\in V$.  Let $\sigma: G \rightarrow \bigoplus_{v\in V} G_v$ be the obvious surjective map.  Eda proved in \cite{E} that $\bigoplus_{v\in V} G_v$ is n-slender.  Thus by the previous lemma we will be done if we can show that $\ker(\sigma)$ is n-slender.  We prove in fact that $\ker(\sigma)$ is u.m.

The length function $l$ on $G$ described above restricts to a length function on $\ker(\sigma)$.  We show that for $g\in \ker(\sigma)-\{1\}$ we have $l(g^2)> l(g)$, which is sufficient to show that $l$ is a uniformly monotone length function.  Let $g\in \ker(\sigma)-\{1\}$ be given and $g =_G g_0g_1\cdots g_{n-1}$ be a reduced word representation as described in Lemma \ref{Green}.  We permute the syllables of the word $g_0g_1\cdots g_{n-1}$ to get a possibly nicer reduced word representation.  To start, suppose that for some $0\leq i<j\leq n-1$ we have that $[g_i, g_{i-1}] = [g_i, g_{i-1}] = \cdots = [g_i, g_0] = 1 = [g_j, g_{j+1}] = \cdots = [g_j, g_{n-1}]$ and $g_i = g_j^{-1}$.  Then permute the syllables and relabel them so that $i = 0$ and $j = n-1$ and $g_0=g_{n-1}^{-1}$.  Perform the same process on the word $g_1g_2\cdots g_{n-2}$, moving a pair of mutual inverses to the front and rear of the word if possible.  Continue this process until it is impossible to go further, so that by relabeling we get a (possibly empty, in case the process can never be performed) initial segment $g_0g_1\cdots g_{k-1}$ and a (possibly empty) terminal segment $g_{n-k} \cdots g_{n-1}$ such that for $0 \leq i<k$ we have $g_i = g_{n-i-1}^{-1}$ and the process cannot be performed on the word $g_{k} g_{k+1}\cdots g_{n-k-1}$.

We next manipulate the word $g_kg_{k+1}\cdots g_{n-k-1}$.  If there exists $k \leq i< j \leq n-k$ such that $[g_i, g_{i-1}] = [g_i, g_{i-2}] = \cdots = [g_{i}, g_k] = 1 = [g_j, g_{j+1}] = \cdots = [g_j, g_{n-k-1}]$ and both $g_i$ and $g_j$ are in the same generating group, then move the syllable $g_i$ to the front and the syllable $g_j$ to the rear of the word $g_kg_{k+1}\cdots g_{n-k-1}$ so that by relabeling we may assume $i=k$ and $j= n-k-1$.  For the word $g_{k+1}g_{k+2}\cdots g_{n-k-2}$ consider whether there exist $k+1\leq i<j\leq n-k-2$ such that $[g_i, g_{i-1}] = [g_i, g_{i-2}] = \cdots = [g_{i}, g_k] = 1 = [g_j, g_{j+1}] = \cdots = [g_j, g_{n-k-1}]$ and both $g_i$ and $g_j$ are in the same generating group.  If so, permute the syllables of the word $g_{k+1}g_{k+2}\cdots g_{n-k-2}$ so that the syllable $g_i$ is now in the front and the syllable $g_j$ is at the rear.  By relabeling we our modified word we may assume that $i = k+1$ and $j = n-2$.  Perform the same process on the word $g_{k+2}g_{k+3}\cdots g_{n-k-3}$, and continue the process until it becomes impossible.  Thus we obtain a (possibly empty, in case the process cannot be performed) initial segment $g_kg_{k+1}\cdots g_{k+p-1}$ and (possibly empty) terminal segment $g_{n-k-p}\cdots g_{n-k-1}$ of the word $g_kg_{k+1}\cdots g_{n-k-1}$ such that all elements of the set $\{g_k, \ldots, g_{k+p-1}\}$ commute with each other and for $k \leq i \leq k+p$ we have that $g_i$ is in the same generating group as $g_{n-i-1}$.  For $k\leq i \leq k+p-1$ let $h_i$ be the element $g_{n-i-1}g_i$ in the generating group containing $g_i$.  By the first process that was performed, we have that $h_i \neq 1$.

Notice that the syllables $g_{k+p-1}$ and $g_{n-k-p}$ cannot be side by side, since otherwise we have that all syllables of $g_kg_{k+1}\cdots g_{n-k-1}$ commute with each other and thus the word $g_0 g_1 \cdots g_{n-1}$ was not reduced.  Thus necessarily $n-k-p> k+p$ and there is a nonempty word $w_0$ in between $g_{k+p-1}$ and $g_{n-k-p}$ such that for each $k \leq i\leq k+p-1$there is a syllable in $w_0$ which does not commute with the syllable $g_i$ (else the word $g_0\cdots g_{n-1}$ was not reduced).  Let $w_1$ be the word $g_k\cdots g_{k+p-1}$, $w_1'$ be the word $g_{n-k-p}\cdots g_{n-k-1}$, $w_1''$ be the word $h_k\cdots h_{k+p-1}$, and $w_2$ be the word $g_0\cdots g_{k-1}$.  We already have that $w_2 w_1w_0w_1'(w_2)^{-1}$ is a reduced word representation for $g$.  The equalities

\begin{center}  $g^2 =_G w_2 w_1w_0w_1'(w_2)^{-1}w_2 w_1w_0w_1'(w_2)^{-1}$

$=_G w_2 w_1w_0w_1'w_1w_0w_1'(w_2)^{-1}$

$=_G w_2w_1w_0w_1''w_0w_1'(w_2)^{-1}$
\end{center}

are clear.

We claim that the word $w_2w_1w_0w_1''w_0w_1'(w_2)^{-1}$ is reduced.  Each of the words $w_0, w_1, w_1', w_1'', w_2, (w_2)^{-1}$ is reduced.  The words $w_0, w_1, w_1', w_2, (w_2)^{-1}$ are reduced since they are subwords of a reduced word.  The word $w_1''$ is reduced since the $h_i$ constituting $w_1''$ are nontrivial, commute with each other, are in the same generating groups as the syllables of $w_1$ (or $w_1'$), and $w_1$ is reduced.  No syllable of the word $w_2$ can be permuted to be next to a syllable of the same generating group in the word $w_1$, the first occurence of $w_0$, or $w_1''$ since the original word $w_2w_1w_0w_1'(w_2)^{-1}$ was reduced and the $h_i$ syllables that constitute the word $w_1''$ are from precisely the same generating groups as those syllables that constitute $w_1$ and $w_1'$.  Also, no syllable of the word $w_2$ can be permuted next to a syllable of the same generating group in the second occurence of $w_0$ since the same is true of the first occurence of $w_0$.  No syllable of $w_2$ can be permuted next to a syllable of the same generating group in the words $w_1'$ and $(w_2)^{-1}$ since the original word $w_2 w_1w_0w_1'(w_2)^{-1}$ was reduced.  That no syllable in $w_1$ can be permuted next to a syllable of the same generating group in any of the words to the right of $w_1$ follows similar lines.  No syllable in the first occurence of $w_0$ can be permuted next to a syllable of the same generating group in $w_1'$ since the subword $w_0w_1$ of the word $w_2 w_1w_0w_1'(w_2)^{-1}$ is reduced.  If the syllable $g_q$ of the first occurence of $w_0$ can be permuted next to a syllable $g_{q'}$ in the second occurence of $w_0$, where $g_q, g_{q'}\in G_v$, then $g_q$ must commute with all the syllables of $w_1$ (and of $w_1'$ and $w_1''$).  By the second process, which was performed on the word $g_k\cdots g_{n-k-1}$, it must be that $g_q=g_{q'}$ and thus in fact $g_q$ commutes with all syllables in the word $g_k\cdots g_{n-k-1}$.  We have that all other syllables of the word $g_k\cdots g_{n-k-1}$ are not in $G_v$ (since $g_k\cdots g_{n-k-1}$ is reduced) and so $\sigma(g_k\cdots g_{n-k-1}) = g_q =\neq 1$ is conjugate to $\sigma(g) = 1$, a contradiction.  The remaining cases are straightforward to check and follow the same lines.

Thus we have that

\begin{center}  $l(g^2) = 2l(w_2) + 3l(w_1) + 2l(w_0) > 2l(w_2) + 2l(w_1) + l(w_0) = l(g)$
\end{center}

since, although $l(w_2)$ and/or $l(w_1)$ might be zero, we demonstrated that $l(w_0)$ is not zero.

\end{proof}

\end{section}

\begin{section}{Thompson's Group}

The group $F$ of R. Thompson is a well-known finitely presented group which satisfies many curious properties, and about which many open problems remain.  The group has no infinitely divisible elements and has the unique extraction of roots property (i.e. if $g^n = h^n$ and $n >0$ then $g = h$).  Thus one might ask the following:

\begin{question*}  Is Thompson's group $F$ n-slender?

\end{question*}

We show that the natural length function defined on diagram groups is not universally monotone.  This does not rule out the possibility of a universally monotone length function, nor the n-slenderness of $F$.

We begin with a discussion of diagram groups (essentially following \cite{GS}), of which $F$ is an example.  Start with an alphabet $X$.  Given two words $u, v$ in the elements of $X$ a \textbf{cell}  $(u \rightarrow v)$ is a directed planar graph consisting of exactly two directed paths with the same initial and terminal vertices and which share no other vertices, which share no edges, and with the two paths labeled by the words $u$ and $v$.  The path labeled by $u$ is called the top path and that labeled by $v$ is called the bottom path.  A \textbf{trivial diagram} is a single directed path labeled by a word in the elements of $X$; if that word is $u$ we denote the trivial diagram by $\epsilon(u)$.  For a trivial diagram we say that the path defining the trivial diagram is both the top and the bottom path.
%
%

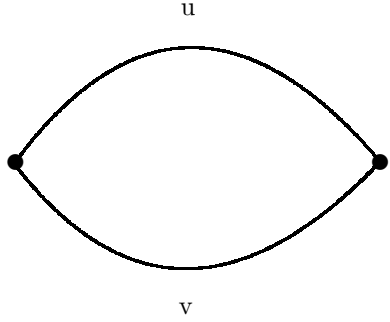
\begin{center}

\unitlength 1mm 
\linethickness{0.7pt}
\ifx\plotpoint\undefined\newsavebox{\plotpoint}\fi 
\begin{picture}(47.75,50)(0,0)
\thicklines
\qbezier(0,28.25)(22.625,58.625)(48.5,28.25)
\qbezier(48.5,28.25)(21.5,0)(0,27.75)
\put(23,48.5){\makebox(0,0)[cc]{u}}
\put(22.625,8.75){\makebox(0,0)[cc]{v}}
\put(0,28.25){\circle*{2}}
\put(48.5,28.25){\circle*{2}}
\end{picture}

\end{center}

%
%

We declare that diagrams are defined only up to planar isotopy and that cells and trivial diagrams are diagrams.  Thus the diagrams that we have so far have an initial and a terminal vertex, that is, any two maximal paths begin and end at the same vertices. Also, each diagram has a top and a bottom path.  Given a diagram $\Delta$ we let $\iota(\Delta)$ and $\tau(\Delta)$ denote the initial and terminal vertices respectively.  All diagrams will similarly have an initial and terminal vertex as well as a top and a bottom path.  In addition to the cells and the trivial diagrams, we close the collection of diagrams under the following three operations:

1.  Addition.  Given two diagrams $\Delta_0$ and $\Delta_1$ we let $\Delta_0 + \Delta_1$ be the planar graph created by identifying $\tau(\Delta_0)$ with $\iota(\Delta_1)$.  Thus the top path of $\Delta_0 + \Delta_1$ is the concatenation of the top paths of $\Delta_0$ and $\Delta_1$, and similarly for the bottom paths.  Also we have $\iota(\Delta_0 + \Delta_1) = \iota(\Delta_0)$ and $\tau(\Delta_0 + \Delta_1) = \tau(\Delta_1)$.  The operation $+$ is clearly associative.  If $u = x_0x_1\cdots x_k$ then we may write $\epsilon(u) = \epsilon(x_0) + \epsilon(x_1) + \cdots + \epsilon(x_k)$.

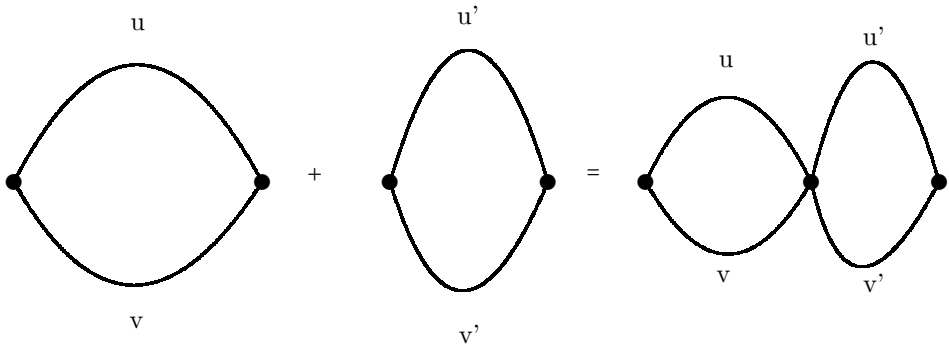
\unitlength 1mm 
\linethickness{0.4pt}
\ifx\plotpoint\undefined\newsavebox{\plotpoint}\fi 
\begin{picture}(124.553,64)(0,0)
\thicklines
\qbezier(0,29)(16.214,60.125)(33,29)
\qbezier(33,29)(15.427,1.5)(0,29)
\put(16.477,50){\makebox(0,0)[cc]{u}}
\put(16.301,10.25){\makebox(0,0)[cc]{v}}
\qbezier(50,29)(60.401,64)(71,29)
\qbezier(71,29)(58.826,0)(50,29)
\put(60.401,51.25){\makebox(0,0)[cc]{u'}}
\put(60.576,8.75){\makebox(0,0)[cc]{v'}}

\put(77, 30){\makebox(0,0)[cc]{$=$}}
\qbezier(84,29)(94.701,51.5)(106,29)
\qbezier(106,29)(94.876,9.75)(84,29)
\qbezier(123,29)(113.864,60.875)(106,29)
\qbezier(106,29)(111.064,6.375)(123,29)
\put(94.701,45){\makebox(0,0)[cc]{u}}
\put(114.302,48.5){\makebox(0,0)[cc]{u'}}
\put(94.352,16.5){\makebox(0,0)[cc]{v}}
\put(114.302,15.5){\makebox(0,0)[cc]{v'}}
\put(0,29){\circle*{2}}
\put(33,29){\circle*{2}}
\put(50,29){\circle*{2}}
\put(71,29){\circle*{2}}
\put(84,29){\circle*{2}}
\put(106,29){\circle*{2}}
\put(123,29){\circle*{2}}

\put(40, 30){\makebox(0,0)[cc]{$+$}}
\end{picture}

\newpage

2.  Multiplication.  Given diagrams $\Delta_0$ and $\Delta_1$ such that the bottom path of $\Delta_0$ has the same label as the top path of $\Delta_1$ we let $\Delta_0 \circ \Delta_1$ be the planar graph obtained by identifying the bottom path of $\Delta_0$ with the top of $\Delta_1$.  Thus under the identification we have $\iota(\Delta_0\circ\Delta_1) = \iota(\Delta_0) = \iota(\Delta_1)$ (and similarly for $\tau$), the top of $\Delta_0\circ \Delta_1$ is the top of $\Delta_0$ and the bottom of $\Delta_0 \circ \Delta_1$ is the bottom of $\Delta_1$.

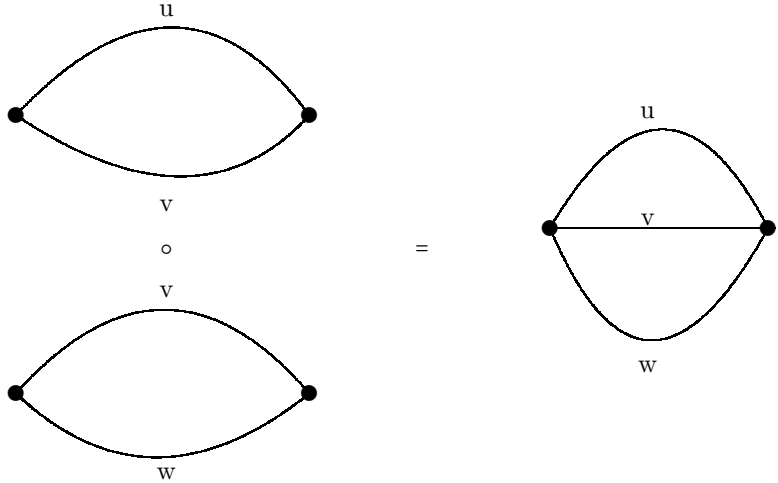
\unitlength 1mm 
\linethickness{0.4pt}
\ifx\plotpoint\undefined\newsavebox{\plotpoint}\fi 
\begin{picture}(112.276,77.375)(0,0)
\qbezier(39,54)(21.901,77.375)(0,54)
\qbezier(0,54)(24.276,37.75)(39,54)
\put(20,68){\makebox(0,0)[cc]{u}}
\put(20,42){\makebox(0,0)[cc]{v}}
\put(20, 36){\makebox(0,0)[cc]{$\circ$}}
\put(54, 36){\makebox(0,0)[cc]{$=$}}
\qbezier(71,39)(86.401,65.25)(100,39)
\qbezier(0,17)(19.901,39.25)(39,17)
\qbezier(39,17)(17.901,0)(0,17)
\put(20,30.625){\makebox(0,0)[cc]{v}}
\put(20,6.375){\makebox(0,0)[cc]{w}}
\qbezier(71,39)(83.526,9.125)(100,39)
\put(71,39){\line(1,0){30}}
\put(84,54.375){\makebox(0,0)[cc]{u}}
\put(84,40.125){\makebox(0,0)[cc]{v}}
\put(84,20.625){\makebox(0,0)[cc]{w}}
\put(0,54){\circle*{2}}
\put(39,54){\circle*{2}}
\put(39,17){\circle*{2}}
\put(0,17){\circle*{2}}
\put(71,39){\circle*{2}}
\put(100,39){\circle*{2}}
\end{picture}

3.  Inversion.  Given a diagram $\Delta$ we define $\Delta^{-1}$ to be the diagram obtained by flipping the diagram $\Delta$ about a horizontal line, so that the top path becomes the bottom path and vice versa.

By definition, the class of diagrams (over $X$) is built out of cells and trivial diagrams using the above operations.  If we wish, we can restrict our attention to those diagrams which are built only from trivial diagrams and cells in a set $P$ and the three operations above and let $D(P)$ denote this class.  If a diagram $\Delta$ has two cells such that the top of the second is identified with the bottom of the first, and the first and second cells are inverses of each other, then we call this pair of cells a \textbf{dipole}.  Notice that if one eliminates the two cells from the diagram and identifies the top of the first cell with the bottom of the second, then we have a new diagram $\Delta'$ and say that $\Delta$ and $\Delta'$ are equivalent.  This induces an equivalence relation on $D(P)$ by making the relation reflexive, symmetric and transitive.  We say $\Delta$ is \textbf{reduced} in case $\Delta$ has no dipoles, and note that every diagram is equivalent to a unique reduced diagram (see \cite{GS}).

Given a word $u$ and a collection of cells $P$ we let $D(P, u)$ denote the collection of reduced diagrams built by using cells in $P$, trivial diagrams, and the above three operations and whose top and bottom paths are labeled by the word $u$.  This forms a group by letting the binary operation be given by $\Delta_0 \Delta_1 = \Delta$, where $\Delta$ is the reduced diagram equivalent to $\Delta_0 \circ \Delta_1$ (see \cite{GS} for a proof).  The identity element and the inverse operation are clear.

Now we state a characterization of $F$ as a diagram group.  We shall use this as our working definition of $F$, and the isomorphism of $F$ with the group we define is given in \cite{GS}.  Letting $X = \{x\}$ and $P = \{(x^2 \rightarrow x)\}$ it is shown that $D(P, x) \simeq F$.  Given a reduced diagram $\Delta \in F$ we let $l(\Delta)$ be the number of cells in $\Delta$.  It is easy to check that $l$ is a length function.

We now show that $l$ is not universally monotone.  Letting $n>1$ be given we give an example of a diagram $\Delta$ such that $l(\Delta)> l(\Delta^n)$.  Since the alphabet $X$ includes only the letter $x$, we may assume that each arc in our diagrams is labeled by the letter $x$ as read from left to right.  Let $\rho$ denote the $(1, 2)$ diagram.  Let $\theta$ be the reduced diagram with $4$ cells given below.
\vspace{.75in}

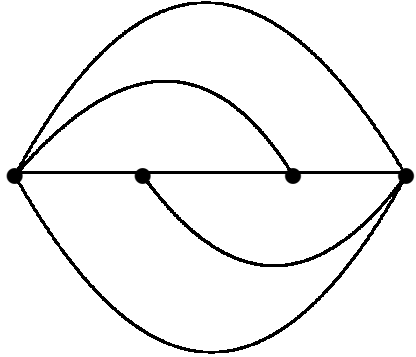
\begin{center}

\unitlength 1mm 
\linethickness{0.7pt}
\ifx\plotpoint\undefined\newsavebox{\plotpoint}\fi 
\begin{picture}(50,55)(0,0)
\put(0,47.375){\line(1,0){51.75}}
\qbezier(0,47)(21.5,72.125)(37,47)
\qbezier(52,47)(34.375,23)(17,47)
\qbezier(0,47)(24.75,93)(52,47)
\qbezier(52,47)(26.25,0)(0,47)
\put(52,47){\circle*{2}}
\put(37,47){\circle*{2}}
\put(17,47){\circle*{2}}
\put(0,47){\circle*{2}}
\end{picture}

\end{center}

%
%

For $n \in \mathbb{Z}$ let $\theta^n$ denote the reduced diagram associated with multiplying $\theta$ with itself $n$ times.  The diagrams for $\theta^2$ and respectively for $\theta^m$ for $m \geq 1$ are straightforward to compute and are pictured below having $6$ and $2 +2m$ cells, resp.

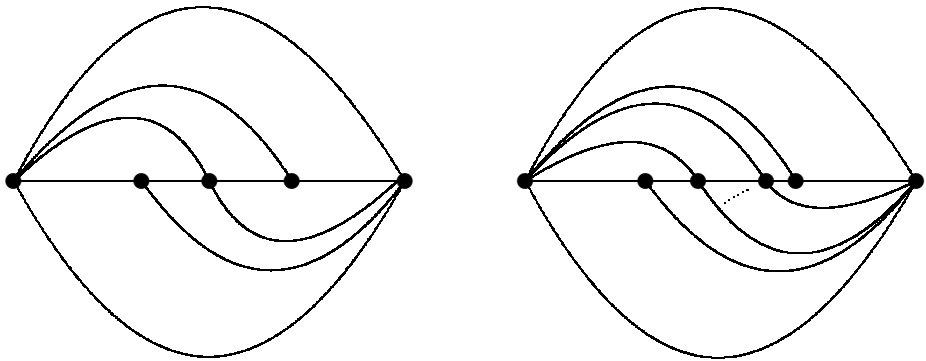
\begin{center}

\unitlength 1mm 
\linethickness{0.4pt}
\ifx\plotpoint\undefined\newsavebox{\plotpoint}\fi 
\begin{picture}(120.531,80)(0,0)
\put(68,47){\line(1,0){51}}
\put(0,47){\line(1,0){51}}
\qbezier(68,47)(88.531,72.125)(104,47)
\qbezier(0,47)(21.031,72.375)(37,47)
\qbezier(120,47)(101.406,23)(84,47)
\qbezier(52,47)(33.906,23.25)(17,47)
\qbezier(68,47)(91.781,93)(120,47)
\qbezier(0,47)(24.281,93.25)(52,47)
\qbezier(120,47)(93.281,0)(68,47)
\qbezier(52,47)(25.781,.25)(0,47)
\put(68,47){\circle*{2}}
\put(0,47){\circle*{2}}
\put(84,47){\circle*{2}}
\put(17,47){\circle*{2}}
\put(104,47){\circle*{2}}
\put(37,47){\circle*{2}}
\put(120,47){\circle*{2}}
\put(52,47){\circle*{2}}
\qbezier(0,47)(17.656,63.875)(26,47)
\qbezier(26,47)(33.781,31)(51,47)
\put(26,47){\circle*{2}}
\qbezier(68,47)(84.281,57.375)(91,47)
\qbezier(91,47)(103.531,27.75)(120,47)
\put(91,47){\circle*{2}}
\qbezier(68,47)(86.656,67.625)(100,47)
\qbezier(100,47)(105.156,39.875)(120,47)
\put(100,47){\circle*{2}}

\multiput(94.461,44.055)(.6,.35){6}{{\rule{.4pt}{.4pt}}}
\end{picture}

\end{center}

Define $\Delta_n$ to be $\underbrace{\theta^{-1} + \theta^{-1}\cdots + \theta^{-1}}_{n-1\text{times}}+\theta^{n-1}$.  Select $k\in \omega$ large enough that $2n^2< 2k +2$.  Let $k_1, k_2 >n+1$ be such that $1+k_1+nk+k_2 = 2^m$ for some $m\in \omega$.  Let $\Psi$ be the diagram pictured below with top path of length $1$ and bottom path of length $2^m +2$.

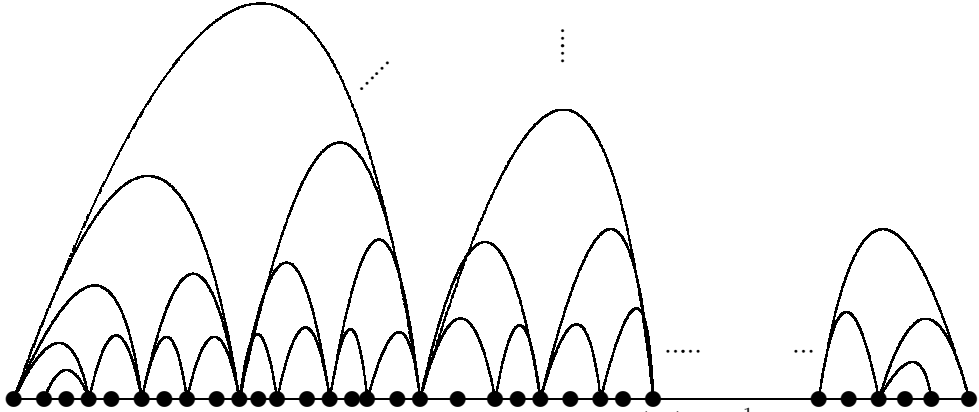
\unitlength 1mm 
\linethickness{0.4pt}
\ifx\plotpoint\undefined\newsavebox{\plotpoint}\fi 
\begin{picture}(129.031,60)(0,0)
\put(0,0){\line(1,0){127}}
\qbezier(0,0)(7,14.825)(10,0)
\qbezier(4,0)(7,7.65)(10,0)
\qbezier(10,0)(13.675,16.838)(17,0)
\qbezier(17,0)(20.575,16.488)(23,0)
\qbezier(23,0)(26.8,16.488)(30,0)
\qbezier(30,0)(32.275,17.188)(35,0)
\qbezier(35,0)(39.025,18.938)(42,0)
\qbezier(42,0)(44.95,18.5)(47,0)
\qbezier(47,0)(51.925,17.713)(54,0)
\qbezier(54,0)(59.8,21.387)(64,0)
\qbezier(64,0)(67.45,19.55)(70,0)
\qbezier(70,0)(75.625,19.813)(78,0)
\qbezier(0,0)(13,30.137)(17,0)
\qbezier(17,0)(24.175,33.287)(30,0)
\qbezier(30,0)(36.475,36.262)(42,0)
\qbezier(42,0)(49.075,42.387)(54,0)
\qbezier(54,0)(63.25,41.688)(70,0)
\qbezier(0,0)(20.65,59.275)(30,0)
\qbezier(30,0)(44.65,68.2)(54,0)
\qbezier(78,0)(84.025,24.012)(85,0)
\qbezier(70,0)(81.175,45.1)(85,0)
\qbezier(54,0)(76.6,76.95)(85,0)
\qbezier(0,0)(38.65,105.212)(54,0)

\put(47, 42){\makebox(0,0)[cc]{$\udots$}}
\put(49, 44){\makebox(0,0)[cc]{$\udots$}}

\put(73,46){\makebox(0,0)[cc]{$\vdots$}}
\put(73,48){\makebox(0,0)[cc]{$\vdots$}}

\put(88, 6){\makebox(0,0)[cc]{$\cdots$}}
\put(90, 6){\makebox(0,0)[cc]{$\cdots$}}

\put(105, 6){\makebox(0,0)[cc]{$\cdots$}}

\qbezier(127,0)(121.45,21.3)(115,0)
\qbezier(115,0)(110.35,23.05)(107,0)
\qbezier(115,0)(120.4,9.75)(122,0)
\qbezier(107,0)(114.025,45.1)(127,0)
\put(0,0){\circle*{2}}
\put(4,0){\circle*{2}}
\put(7,0){\circle*{2}}
\put(10,0){\circle*{2}}
\put(13,0){\circle*{2}}
\put(17,0){\circle*{2}}
\put(20,0){\circle*{2}}
\put(23,0){\circle*{2}}
\put(27,0){\circle*{2}}
\put(30,0){\circle*{2}}
\put(32.5,0){\circle*{2}}
\put(35,0){\circle*{2}}
\put(39,0){\circle*{2}}
\put(42,0){\circle*{2}}
\put(45,0){\circle*{2}}
\put(47,0){\circle*{2}}
\put(51,0){\circle*{2}}
\put(54,0){\circle*{2}}
\put(59,0){\circle*{2}}
\put(64,0){\circle*{2}}
\put(67,0){\circle*{2}}
\put(70,0){\circle*{2}}
\put(74,0){\circle*{2}}
\put(78,0){\circle*{2}}
\put(81,0){\circle*{2}}
\put(85,0){\circle*{2}}
\put(107,0){\circle*{2}}
\put(111,0){\circle*{2}}
\put(115,0){\circle*{2}}
\put(118.5,0){\circle*{2}}
\put(122,0){\circle*{2}}
\put(127,0){\circle*{2}}
\end{picture}

Let $\chi$ be the diagram $\rho + \epsilon{k_1}+ \underbrace{\Delta_n + \Delta_n + \cdots + \Delta_n}_{k \text{times}} + \epsilon(k_2) + \rho^{-1}$ and finally let $\Delta = \Psi \circ \chi \circ \circ \Psi^{-1}$.  It is straightforward to check that $\Delta$ is a reduced diagram with top and bottom path of length $1$, so $\Delta$ is in the Thompson group $F$.  To compute $\l(\Delta)$ we note that each $\Delta_n$ has $2+6(n-1)$ cells in it, so that $\chi$ has $1 + k(2 + 6(n-1)) +1$ cells in it.  The diagram $\Psi$ contains $2^m +1$ cells, so that $\l(\Delta) = 1 + k(2 + 6(n-1)) +1 + 2(2^m +1)= 4-4k +6kn+2^{m+1}$.  Clearly $\Delta^n$ is equivalent to $\Psi\circ  \chi^n \circ \Psi$.  We count the number of cells in $\Psi\circ  \chi^n \circ \Psi$ to obtain an upper bound on $l(\Delta^n)$.   First of all we have $2(2^m+1)$ cells in $\Psi$ and $\Psi^{-1}$ combined.  The reduced diagram $\chi^n$ is given by the equality 
\begin{center}
$\chi^n = \Gamma_1+ \epsilon(k_1 - n)+ \theta^{-1}+ \theta^{-2}+ \cdots + \theta^{1-n} + \epsilon((k-1)n +2) + \theta + \theta^{2} +\cdots + \theta^{n-1} + \epsilon(k_2) + \Gamma_2$ 
\end{center}
where $\Gamma_1$ is the diagram with $n$ cells pictured below and $\Gamma_2$ is the rotation of $\Gamma_1$ by $180$ degrees.
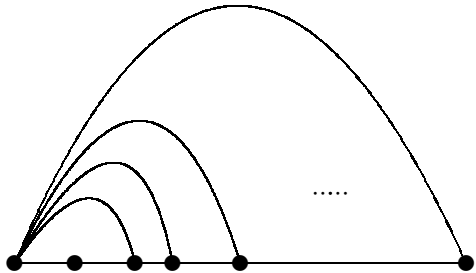
\begin{center}

\unitlength 1mm 
\linethickness{0.4pt}
\ifx\plotpoint\undefined\newsavebox{\plotpoint}\fi 
\begin{picture}(61.781,50)(0,0)
\put(0,0){\line(1,0){59.25}}
\qbezier(0,0)(29.281,68.5)(60,0)
\qbezier(0,0)(11.781,17.25)(16,0)
\qbezier(0,0)(15.781,26.75)(21,0)
\qbezier(0,0)(18.156,37.875)(30,0)
\put(0,0){\circle*{2}}
\put(16,0){\circle*{2}}
\put(8,0){\circle*{2}}
\put(21,0){\circle*{2}}
\put(30,0){\circle*{2}}
\put(60,0){\circle*{2}}

\put(43,9){\makebox(0,0)[cc]{$\cdots$}}
\put(41,9){\makebox(0,0)[cc]{$\cdots$}}
\end{picture}

\end{center}

Thus $\chi^n$ has $n + (4+6+ \cdots + (2+2(n-1))) + (4+6+ \cdots + (2+2(n-1))) + n = 2n + 2(n^2 - n) = 2n^2$ cells.  Hence $l(\Delta^n) \leq 2n^2 + 2(2^m+1)< 4-4k +6kn +2^{m+1} = l(\Delta)$ by our choice of $k$, as desired.

\end{section}

\begin{section}{A Family of Non-Examples}

Although the uniformly monotone condition gives a very nice sufficient condition for n-slenderness, it is not a necessary condition.  We give an example of a very basic family of groups that are n-slender but which are not uniformly monotone.

The countable slender abelian groups have the following criterion (see \cite{F}):

\begin{lemma}  A countable abelian group $A$ is slender if and only if $A$ is torsion-free and reduced (i.e. $\bigcap_{m = 1}^{\infty} mA = \bigcap_{m = 1}^{\infty}\{ma: a\in A\}$ is trivial.)
\end{lemma}

\begin{theorem}  The groups $BS(1, n)$ are n-slender but not u.m. for $n>1$.  Moreover, these groups are HNN extensions of the u.m. group $\mathbb{Z}$.
\end{theorem}

\begin{proof}  Recall that $BS(1, n) = \langle a,b| bab^{-1} = a^n\rangle$.  The retraction map $q: BS(1,n) \rightarrow \langle b\rangle$ defined by $b\mapsto b$ and $a\mapsto 1$ has kernel which is easily seen to be isomorphic to the additive group of the $n$-adic rational numbers, $\mathbb{Z}[\frac{1}{n}]$.

The group $\mathbb{Z}[\frac{1}{n}]$ is clearly torsion-free and countable.  To see that $\mathbb{Z}[\frac{1}{n}]$ is reduced we notice that an element $\frac{r}{n^k}$ with $r\in \mathbb{Z}$ is not a $p$-th power for any $p$ that divides neither $n$ nor $r$.  Thus $\mathbb{Z}[\frac{1}{n}]$ is slender and therefore n-slender, and since $\mathbb{Z} \simeq \langle b\rangle$ is also n-slender we know that $BS(1, n)$ is n-slender as an n-slender by n-slender group (Lemma \ref{extension}).

However we know that $BS(1, n)$ cannot be u.m. since it contains a subgroup isomorphic to $\mathbb{Z}[\frac{1}{n}]$.  Each element of $\mathbb{Z}[\frac{1}{n}]$ is infinitely divisible, and a u.m. group cannot have any infinitely divisible elements besides the identity element.

\end{proof}

\end{section}

\bibliographystyle{amsplain}

\end{document}